\documentclass[a4paper,12pt,intlimits,oneside]{amsart}

\usepackage{enumerate}
\usepackage{mathrsfs}
\usepackage{amsfonts,amsmath}

\usepackage{latexsym,amssymb,amsthm}

\newcommand{\comment}[1]{}

\newcommand{\R}{{\mathbb R}}

\def\H{{\mathcal H}}

\newcounter{rea}
\newcounter{rek}
\newcounter{res}
\begin{document}
\title[]{Norm of the Hausdorff operator on the real Hardy space $H^1(\R)$}         % Enter your title between curly braces
\author{Ha Duy Hung}    
\address{High School for Gifted Students, Hanoi National University of Education, 136 Xuan Thuy, Hanoi, Vietnam} 
\email{{\tt hunghaduy@gmail.com}}
\author{Luong Dang Ky}
\address{Department of Mathematics, Quy Nhon University, 
170 An Duong Vuong, Quy Nhon, Binh Dinh, Viet Nam} 
\email{{\tt luongdangky@qnu.edu.vn}}
\author{Thai Thuan Quang}
\address{Department of Mathematics, Quy Nhon University, 
170 An Duong Vuong, Quy Nhon, Binh Dinh, Viet Nam} 
\email{thaithuanquang@qnu.edu.vn}

\begin{abstract} Let $\varphi$ be a nonnegative integrable function on $(0,\infty)$. It is well-known that the Hausdorff operator $\H_\varphi$ generated by $\varphi$ is bounded on the real Hardy space $H^1(\R)$. The aim of this paper is to give the exact norm of $\H_\varphi$. More precisely, we prove that
$$\|\H_\varphi\|_{H^1(\R)\to H^1(\R)}= \int_0^\infty \varphi(t)dt.$$
\end{abstract}

\maketitle
\newtheorem{theorem}{Theorem}[section]
\newtheorem{lemma}{Lemma}[section]
\newtheorem{proposition}{Proposition}[section]
\newtheorem{remark}{Remark}[section]
\newtheorem{corollary}{Corollary}[section]
\newtheorem{definition}{Definition}[section]
\newtheorem{example}{Example}[section]
\numberwithin{equation}{section}
\newtheorem{Theorem}{Theorem}[section]
\newtheorem{Lemma}{Lemma}[section]
\newtheorem{Proposition}{Proposition}[section]
\newtheorem{Remark}{Remark}[section]
\newtheorem{Corollary}{Corollary}[section]
\newtheorem{Definition}{Definition}[section]
\newtheorem{Example}{Example}[section]
\newtheorem*{theorema}{Theorem A}

\section{Introduction and main result} 
\allowdisplaybreaks

Let $\varphi$ be a locally integrable function on $(0,\infty)$. The {\it Hausdorff operator} $\H_\varphi$ is defined for suitable functions $f$ by
$$\H_\varphi f(x)=\int_0^\infty f\left(\frac{x}{t}\right) \frac{\varphi(t)}{t} dt.$$

The Hausdorff operator is an interesting operator in harmonic analysis. There are many classical operators in analysis which are special cases of the Hausdorff operator if one chooses suitable kernel functions $\varphi$, such as the classical Hardy operator, its adjoint operator, the Ces\`aro type operators, the Riemann-Liouville
fractional integral operator,... See the survey article \cite{Li} and the references therein. In the recent years, there is an increasing interest on the study of boundedness of the Hausdorff operator on the real Hardy spaces, see for example \cite{An, CFZ, FL, Li, LM1, LM2, LM3, Mo, RF}.

Let $\Phi$ be a function in the Schwartz space $\mathcal S(\R)$ satisfying $\int_{\R}\Phi(x)dx\ne 0$. Set $\Phi_t(x):= t^{-1}\Phi(x/t)$. Following Fefferman and Stein \cite{FS, St}, we define the {\it real Hardy space} $H^1(\R)$ as the space of functions $f\in L^1(\R)$ such that 
$$\|f\|_{H^1(\R)} := \left\|M_{\Phi} f\right\|_{L^1(\R)}<\infty,$$
where $M_{\Phi} f$ is the {\it smooth maximal function} of $f$ defined by
$$M_{\Phi}f(x)= \sup_{t>0}|f*\Phi_t(x)|,\quad x\in\R.$$

Remark that $\|\cdot\|_{H^1(\R)}$ defines a norm on $H^1(\R)$, whose size depends on the choice of $\Phi$, but the space $H^1(\R)$ does not depend on this choice.

Let $\varphi$ be a nonnegative function in $L^1_{\rm loc}(0,\infty)$. Although, it was shown in \cite{FL} that $\H_\varphi$ is bounded on $H^1(\R)$ if and only if $\varphi\in L^1(0,\infty)$, the exact norm $\|\H_\varphi\|_{H^1(\R)\to H^1(\R)}$ is still unknown.

\vskip 0.2cm

Our main result is as follows.

\begin{theorem}\label{main theorem}
	Let $\varphi$ be a nonnegative function in $L^1(0,\infty)$. Then 
	$$\|\H_\varphi\|_{H^1(\R)\to H^1(\R)}= \int_0^\infty \varphi(t)dt.$$
\end{theorem}

In Theorem \ref{main theorem}, it should be pointed out that the norm of the Hausdorff operator $\H_\varphi$ ($\int_0^\infty \varphi(t)dt$) does not depend on the choice of the above function $\Phi$. Moreover, it still holds when the above norm $\|\cdot\|_{H^1(\R)}$ is replaced by
$$\|f\|_{H^1(\R)}:= \|f\|_{L^1(\R)} + \|H(f)\|_{L^1(\R)},$$
where $H(f)$ is the {\it Hilbert transform} $H$ of $f\in L^1(\R)$ defined by
$$H(f)(x)= \frac{1}{\pi}\; {\rm p.v.}\int_{-\infty}^{\infty}\frac{f(x-y)}{y}dy, \quad {\rm a.e.} \; x\in\R.$$
See the last section for details.

\begin{corollary}
	Let $\varphi\in L^1(0,\infty)$. Then $\H_\varphi$ is bounded on $H^1(\R)$, moreover,
	$$\left|\int_0^\infty \varphi(t)dt\right|\leq \|\H_\varphi\|_{H^1(\R)\to H^1(\R)} \leq \int_0^\infty |\varphi(t)|dt.$$
\end{corollary}

Throughout the whole article, we denote by $C$ a positive constant which is independent of the main parameters, but it may vary from line to line. The symbol $A \lesssim B$  means that $A\leq C B$. If $A \lesssim B$ and $B\lesssim A$, then we  write $A\sim B$.  For any $E\subset \R$, we denote by $\chi_E$ its characteristic function.

%%%%%%%%%%%%%%%%%%%%%%%%%%%%%%%%%%%%%%%%%%%%%%%%%%%%%%%%%%%%%%%%%%

\section{Proof of Theorem \ref{main theorem}}

Let $P$ be the Poisson kernel on $\R$, that is, $P(x)=\frac{1}{x^2+1}$ for all $x\in\R$. For any $t>0$, set $P_t(x):= \frac{t}{x^2+t^2}$.  The {\it Poisson maximal function} $M_P f$ of a function $f\in L^1(\R)$ is then defined by
$$M_P f(x) = \sup_{t>0} |P_t*f(x)|,\quad x\in\R.$$

Let $\mathbb C_+$ be the upper half-plane in the complex plane. The Hardy space $\mathcal H_a^1(\mathbb C_+)$ is defined as the set of all holomorphic functions $F$ on $\mathbb C_+$ such that
$$\|F\|_{\mathcal H_a^1(\mathbb C_+)}:= \sup_{y>0} \int_{-\infty}^{\infty} |F(x+iy)| dx <\infty.$$

The following two lemmas are classical and can be found in \cite{Du, Ga, St}.

\begin{lemma}\label{some characterizations of real Hardy spaces}
	Let $f\in L^1(\R)$. Then the following conditions are equivalent:
	\begin{enumerate}[\rm (i)]
		\item $f\in H^1(\R)$.
		\item $H(f)\in L^1(\R)$.
		\item $M_P f\in L^1(\R)$.
	\end{enumerate}
	Moreover, in that case, 
	$$\|f\|_{H^1(\R)}\sim \|f\|_{L^1(\R)} + \|H(f)\|_{L^1(\R)} \sim \|M_P f\|_{L^1(\R)}.$$
\end{lemma}

\begin{lemma}\label{boundary value function}
	Let $F\in \H_a^1(\mathbb C_+)$. Then the boundary value function $f$ of $F$, which is defined by
		$$f(x)=\lim_{y\to 0} F(x+iy), \quad\mbox{a.e.}\; x\in \R,$$
	is in $H^1(\R)$. Moreover, 
	$$\|f\|_{H^1(\R)}\sim \|f\|_{L^1(\R)}= \|F\|_{\H^1_a(\mathbb C_+)}$$
	and	$F(x+iy)= P_y*f(x)$ for all $x+iy\in\mathbb C_+$.
\end{lemma}

In order to prove Theorem \ref{main theorem}, we also need the following key lemma.

\begin{lemma}\label{key lemma}
	Let $\varphi$ be a nonnegative function in $L^1(0,\infty)$. Then
	\begin{enumerate}[\rm (i)]
		\item $\H_\varphi$ is bounded on $H^1(\R)$, moreover,
		$$\|\H_\varphi\|_{H^1(\R)\to H^1(\R)}\leq \int_0^\infty \varphi(t)dt.$$
		\item If supp $\varphi\subset [0,1]$, then
		$$\|\H_\varphi\|_{H^1(\R)\to H^1(\R)}= \int_0^1 \varphi(t)dt.$$
	\end{enumerate}
\end{lemma}

\begin{proof}
(i)	For any $f\in H^1(\R)$, by the Fubini theorem, we have
	\begin{eqnarray*}
		M_\Phi(\H_\varphi f)(x) &=& \sup_{r>0} \left|\int_{\R} dy \int_0^\infty \frac{1}{r}\Phi\left(\frac{x-y}{r}\right) f\left(\frac{y}{t}\right)\frac{\varphi(t)}{t}dt\right|\\
		&=& \sup_{r>0} \left|\int_0^\infty \Phi_{r/t}*f\left(\frac{x}{t}\right)\frac{\varphi(t)}{t} dt\right|\\
		&\leq&  \H_\varphi(M_\Phi f)(x)
	\end{eqnarray*}
	for all $x\in \R$. Hence,
	\begin{eqnarray*}
		\|\H_\varphi f\|_{H^1(\R)} &=& \|M_\Phi(\H_\varphi f)\|_{L^1(\R)}\\
		&\leq& \int_{\R} dx \int_0^\infty M_\Phi f\left(\frac{x}{t}\right) \frac{\varphi(t)}{t}dt \\
		&=& \int_0^\infty \varphi(t)dt \|M_\Phi(f)\|_{L^1(\R)}= \int_0^\infty \varphi(t)dt \|f\|_{H^1(\R)}.
	\end{eqnarray*}
	This proves that $\H_\varphi$ is bounded on $H^1(\R)$, moreover,
	\begin{equation}\label{key lemma, 1}
	\|\H_\varphi\|_{H^1(\R)\to H^1(\R)}\leq \int_0^\infty \varphi(t)dt.
	\end{equation}	
(ii)	Let $\delta\in (0,1)$ be arbitrary. By (\ref{key lemma, 1}), we see that
	$$\|\H_{\varphi_\delta}\|_{H^1(\R)\to H^1(\R)}\leq \int_0^\infty \varphi_\delta(t)dt = \int_{\delta}^{1} \varphi(t)dt<\infty$$
	and 
	\begin{equation}\label{key lemma, 2}
		\|\H_{\varphi}- \H_{\varphi_\delta}\|_{H^1(\R)\to H^1(\R)}\leq \int_0^\infty[\varphi(t)- \varphi_\delta(t)]dt=  \int_0^\delta \varphi(t)dt<\infty,
	\end{equation}
	where $\varphi_\delta(t):= \varphi(t)\chi_{[\delta,1]}(t)$ for all $t\in (0,\infty)$.
	
	For any $\varepsilon>0$, we define the function $F_\varepsilon:\mathbb C_+\to\mathbb C$ by
	$$F_\varepsilon(z)=\frac{1}{(z+i)^{1+\varepsilon}}$$
	where $\zeta^{1+\varepsilon}= |\zeta|^{1+\varepsilon} e^{i(1+\varepsilon)\arg \zeta}$ for all $\zeta\in\mathbb C$. Then, by Lemma \ref{boundary value function},
	\begin{equation}\label{key lemma, 3}
		\|f_\varepsilon\|_{H^1(\R)}\sim \|F_\varepsilon\|_{\H^1_a(\mathbb C_+)}= \int_{-\infty}^{\infty} \frac{1}{\sqrt{x^2+1}^{1+\varepsilon}}dx <\infty,
	\end{equation}
	where $f_\varepsilon$ is the boundary value function of $F_\varepsilon$.

	For all $z=x +iy\in \mathbb C_+$, by the Fubini theorem and Lemma \ref{boundary value function}, we get
	\begin{eqnarray*}
		&&P_y*\Big(\H_{\varphi_\delta}(f_\varepsilon)- f_{\varepsilon} \int_0^\infty \varphi_\delta(t)dt\Big)(x)\\
		&=& \int_0^\infty \frac{1}{(\frac{z}{t}+i)^{1+\varepsilon}}\frac{\varphi_\delta(t)}{t}dt- \frac{1}{(z+i)^{1+\varepsilon}}\int_0^\infty \varphi_\delta(t)dt \\
		&=&\int_{\delta}^{1} [\phi_{\varepsilon,z}(t)- \phi_{\varepsilon,z}(1)]\varphi(t)dt,
	\end{eqnarray*}
	where $\phi_{\varepsilon,z}(t):= \frac{t^\varepsilon}{(z+ti)^{1+\varepsilon}}$. For any $t\in [\delta,1]$, a simple calculus gives
	\begin{eqnarray*}
		|\phi_{\varepsilon,z}(t)- \phi_{\varepsilon,z}(1)| &\leq& |t-1| \sup_{s\in [\delta,1]}|\phi'_{\varepsilon,z}(s)|\\
		&\leq& \frac{\varepsilon \delta^{-2}}{\sqrt{x^2+1}^{1+\varepsilon}} + \frac{(1+\varepsilon) \delta^{-2}}{\sqrt{x^2+1}^{2+\varepsilon}}.
	\end{eqnarray*}
	Therefore, by Lemma \ref{some characterizations of real Hardy spaces}, 
	\begin{eqnarray*}
		\left\|\H_{\varphi_\delta}(f_\varepsilon)- f_{\varepsilon} \int_0^\infty \varphi_\delta(t)dt\right\|_{H^1(\R)} &\lesssim&  \left\|M_{P}\Big(\H_{\varphi_\delta}(f_\varepsilon)- f_{\varepsilon} \int_0^\infty \varphi_\delta(t)dt\Big)\right\|_{L^1(\R)}\\
		&\leq&\int_{\delta}^{1} \varphi(t)dt \int_{-\infty}^{\infty} \left[\frac{\varepsilon \delta^{-2}}{\sqrt{x^2+1}^{1+\varepsilon}} + \frac{(1+\varepsilon) \delta^{-2}}{\sqrt{x^2+1}^{2+\varepsilon}}\right] dx.
	\end{eqnarray*}
	This, together with (\ref{key lemma, 3}), yields
	\begin{eqnarray*}
		&&\frac{\left\|\H_{\varphi_\delta}(f_\varepsilon)- f_{\varepsilon}\int_0^\infty \varphi_\delta(t)dt\right\|_{H^1(\R)}}{\|f_\varepsilon\|_{H^1(\R)}}\\
		&\lesssim& \int_{\delta}^{1} \varphi(t)dt \left[\varepsilon \delta^{-2}+ \frac{(1+\varepsilon)\delta^{-2}\int_{-\infty}^{\infty} \frac{1}{x^2+1}dx}{\int_{-\infty}^{\infty} \frac{1}{\sqrt{x^2+1}^{1+\varepsilon}}dx}\right] \to 0
	\end{eqnarray*}
	as $\varepsilon \to 0$. As a consequence,
	$$\int_\delta^1 \varphi(t)dt = \int_0^\infty \varphi_\delta(t)dt \leq \|\H_{\varphi_\delta}\|_{H^1(\R)\to H^1(\R)}.$$
	This, combined with (\ref{key lemma, 2}), allows us to conclude that
	$$\|\H_\varphi\|_{H^1(\R)\to H^1(\R)}\geq \int_0^1 \varphi(t)dt - 2 \int_0^{\delta}\varphi(t)dt\to \int_0^1 \varphi(t)dt$$
	as $\delta\to 0$ since $\int_0^1 \varphi(t)dt<\infty$. Hence, by (\ref{key lemma, 1}),
	$$\|\H_\varphi\|_{H^1(\R)\to H^1(\R)}= \int_0^1 \varphi(t)dt.$$

\end{proof}

Now we are ready to give the proof of Theorem \ref{main theorem}.

\begin{proof}[Proof of Theorem \ref{main theorem}]
	 By Lemma \ref{key lemma},
	\begin{equation}\label{main theorem, 0}
		\|\H_{\varphi}\|_{H^1(\R)\to H^1(\R)}\leq \int_0^\infty \varphi(t)dt.
	\end{equation}

	For any $m>0$, we define $\varphi_m(t)= \varphi(mt)\chi_{(0,1]}(t)$ for all $t\in (0,\infty)$. Then, by Lemma \ref{key lemma}, we see that
	$$\left\|\H_{\varphi_m\left(\frac{\cdot}{m}\right)}\right\|_{H^1(\R)\to H^1(\R)}\leq \int_0^\infty \varphi_m\left(\frac{t}{m}\right)dt=\int_0^m \varphi(t)dt<\infty$$
	and
	\begin{equation}\label{main theorem, 2}
		\left\|\H_\varphi- \H_{\varphi_m\left(\frac{\cdot}{m}\right)}\right\|_{H^1(\R)\to H^1(\R)}\leq \int_{0}^{\infty} \left[\varphi(t)- \varphi_m\left(\frac{t}{m}\right)\right] dt= \int_m^\infty \varphi(t)dt<\infty.
	\end{equation}
		
	Noting that
	$$\left\|f\left(\frac{\cdot}{m}\right)\right\|_{H^1(\R)}= m \|f(\cdot)\|_{H^1(\R)}\quad\mbox{and}\quad \H_{\varphi_m\left(\frac{\cdot}{m}\right)}f= \H_{\varphi_m}f\left(\frac{\cdot}{m}\right)$$
	for all $f\in H^1(\R)$, Lemma \ref{key lemma} yields
	$$\left\|\H_{\varphi_m\left(\frac{\cdot}{m}\right)}\right\|_{H^1(\R)\to H^1(\R)} = m \|\H_{\varphi_m}\|_{H^1(\R)\to H^1(\R)} = m \int_0^1 \varphi(mt)dt= \int_0^m \varphi(t)dt.$$
	Combining this with (\ref{main theorem, 2}) allows us to conclude that
	$$\|\H_\varphi\|_{H^1(\R)\to H^1(\R)}\geq \int_0^\infty \varphi(t)dt - 2 \int_m^\infty \varphi(t)dt \to \int_0^\infty \varphi(t)dt$$
	as $m\to\infty$ since $\int_0^\infty \varphi(t)dt<\infty$. Hence, by (\ref{main theorem, 0}),
	$$\|\H_\varphi\|_{H^1(\R)\to H^1(\R)} = \int_0^\infty \varphi(t)dt,$$
	which ends the proof of Theorem \ref{main theorem}.
	
\end{proof}

\section{Appendix}

The main purpose of this section is to show that the norm of the Hausdorff operator $\H_\varphi$ in Theorem \ref{main theorem} ($\int_0^\infty \varphi(t)dt$) still holds even when one replaces $\|f\|_{H^1(\R)}= \|M_\Phi f\|_{L^1(\R)}$ by some other equivalent norms on $H^1(\R)$. Such norms can be defined via the nontangential maximal functions, atoms, the Hilbert transform,... See Stein's book \cite{St}.

Let $\psi$ be a function in the Schwartz space $\mathcal S(\R)$ satisfying $\int_{\R}\psi(x)dx \ne 0$; or be the Poisson kernel $P$ on $\R$. Then, for $f\in L^1(\R)$,  we define the {\it nontangential maximal function} $\mathcal M_\psi f$ of $f$ by
$$\mathcal M_\psi f(x)=\sup_{|x-y|<t} |\psi_t*f(y)|,\quad x\in\R.$$

A function $a$ is called an $H^1$-atom related to the interval $B$ if
\begin{enumerate}[$\bullet$]
	\item supp $a\subset B$;
	\item $\|a\|_{L^\infty(\R)}\leq |B|^{-1}$;
	\item $\int_{\R} a(x)dx=0$.
\end{enumerate} 
We define the Hardy space $H^1_{\rm at}(\R)$ as the space of functions $f\in L^1(\R)$ which can be written as $f=\sum_{j=1}^\infty \lambda_j a_j$ with $a_j$'s are $H^1$-atoms and $\lambda_j$'s are complex numbers satisfying $\sum_{j=1}^\infty |\lambda_j|<\infty$. The norm on $H^1_{\rm at}(\R)$ is then defined by
$$\|f\|_{H^1_{\rm at}(\R)}:= \inf\left\{\sum_{j=1}^\infty |\lambda_j|: f= \sum_{j=1}^\infty \lambda_j a_j\right\}.$$ 

%For $f\in L^1(\R)$, we define the {\it square function} $S(f)$ by
%$$S(f)(x)=\left(\int_{\Gamma(x)}|\nabla u(y,t)|^2 dydt\right)^{1/2}, \quad x\in \R,$$
%where $\Gamma(x):=\{(y,t)\in \R\times (0,\infty): |y-x|<t\}$ and $u(y,t):= P_t*f(y)$.

The following is classical and can be found in Stein's book \cite{St}.

\begin{theorem}\label{some equivalent characterizations of H1}
	Let $f\in L^1(\R)$. Then the following conditions are equivalent:
	\begin{enumerate}[\rm (i)]
		\item $f\in H^1(\R)$.
		\item $\mathcal M_\psi f\in L^1(\R)$.
		\item $f\in H^1_{\rm at}(\R)$.
		\item $H(f)\in L^1(\R)$.
	\end{enumerate}
	Moreover, in that case,
	$$\|f\|_{H^1(\R)} \sim \|\mathcal M_\psi f\|_{L^1(\R)}\sim \|f\|_{H^1_{\rm at}(\R)}\sim \|f\|_{L^1(\R)} + \|H(f)\|_{L^1(\R)}.$$
\end{theorem}

The main aim of this section is to establish the following.

\begin{theorem}\label{the last theorem}
	Let $\varphi$ be a nonnegative function in $L^1(0,\infty)$. Then 
	$$\|\H_\varphi\|_{(H^1(\R),\|\cdot\|_*)\to (H^1(\R),\|\cdot\|_*)}= \int_0^\infty \varphi(t)dt,$$
	where $\|\cdot\|_*$ is one of the four norms in Theorem \ref{some equivalent characterizations of H1}.
\end{theorem}

\begin{proof}
	By the proofs of Lemma \ref{key lemma} and Theorem \ref{main theorem}, we see that 
	$$\|\H_\varphi\|_{(H^1(\R),\|\cdot\|_*)\to (H^1(\R),\|\cdot\|_*)}\geq \int_0^\infty \varphi(t)dt$$
	for any norm of the four norms in Theorem \ref{some equivalent characterizations of H1}. So, it suffices to show
	\begin{equation}\label{the last theorem, 1}
		\|\H_\varphi\|_{(H^1(\R),\|\cdot\|_*)\to (H^1(\R),\|\cdot\|_*)}\leq \int_0^\infty \varphi(t)dt.
	\end{equation}
		
	{\bf Case 1:} $\|f\|_*= \|f\|_{L^1(\R)}+ \|H(f)\|_{L^1(\R)}$. For any $f\in H^1(\R)$, we have
	$$\|\H_\varphi f\|_{L^1(\R)}\leq \int_0^\infty \varphi(t)dt \|f\|_{L^1(\R)}$$
	and
	$$\|H(\H_\varphi f)\|_{L^1(\R)}= \|\H_\varphi (H(f))\|_{L^1(\R)}\leq \int_0^\infty \varphi(t)dt \|H(f)\|_{L^1(\R)}$$	
	by \cite[Theorems 1 and 3]{LM1}. This implies that (\ref{the last theorem, 1}) holds.
	
	\vskip 0.3cm
	
	{\bf Case 2:} $\|f\|_*=\|f\|_{H^1_{\rm at}(\R)}$. Denote by $BMO(\R)$ the John-Nirenberg space (see \cite{St}) with the norm
	$$\|g\|_{BMO}:= \sup_{B}\frac{1}{|B|}\int_B \Big| g(x) -\frac{1}{|B|}\int_B g(y)dy\Big|<\infty,$$
	where the supremum is taken over all intervals $B\subset \R$. It is well-known that $BMO(\R)$ is the dual space of $H^1(\R)$, moreover,
	$$\|g\|_{BMO}= \sup_{\|f\|_*\leq 1} \left|\int_{\R} f(x)g(x)dx\right|,$$
	where the supremum is taken over all functions $f\in H^1(\R)$ with $\|f\|_*\leq 1$. Therefore, by \cite[Theorem 3]{An} and a standard functional analysis
	argument, 
	$$\|\H_\varphi\|_{(H^1(\R),\|\cdot\|_*)\to (H^1(\R),\|\cdot\|_*)}= \|\H^*_\varphi\|_{BMO\to BMO}\leq \int_0^\infty \varphi(t)dt,$$ 
	where $\H^*_\varphi$ is the conjugated operator of $\H^*_\varphi$ defined on $BMO(\R)$ by
	$$\H^*_\varphi g(x):= \int_0^\infty g(tx)\varphi(t)dt, \quad x\in \R.$$
	
	{\bf Case 3:} $\|f\|_*= \|\mathcal M_\psi f\|_{L^1(\R)}$. For any $f\in H^1(\R)$, the Fubini theorem gives
	\begin{eqnarray*}
		\mathcal M_\psi (\H_\varphi f)(x) &=& \sup_{|y-x|<r}\left|\int_{\R} dz \int_0^\infty \frac{1}{r}\psi\left(\frac{y-z}{r}\right) f\left(\frac{z}{t}\right) \frac{\varphi(t)}{t} dt\right|\\
		&=& \sup_{|y-x|<r}\left| \int_0^\infty \psi_{r/t}*f\left(\frac{y}{t}\right)\frac{\varphi(t)}{t}dt \right|\\
		&\leq& \H_\varphi(\mathcal M_\psi f)(x)
	\end{eqnarray*}
	for all $x\in \R$. Hence,
	\begin{eqnarray*}
		\|\H_\varphi f\|_* &=& \|\mathcal M_\psi (\H_\varphi f)\|_{L^1(\R)}\\
		&\leq& \|\H_\varphi\|_{L^1(\R)\to L^1(\R)} \|\mathcal M_\psi f\|_{L^1(\R)}= \int_0^\infty \varphi(t)dt \|f\|_*,
	\end{eqnarray*}
	which implies that (\ref{the last theorem, 1}) holds, and thus ends the proof of Theorem \ref{the last theorem}.
	
	\end{proof}

Finally, we give a new proof for a known result (see \cite[Theorem 1.2]{FL}).

\begin{theorem}
	Let $\varphi$ be a nonnegative function in $L^1_{\rm loc}(0,\infty)$ satisfying that $\H_\varphi$ is bounded on $H^1(\R)$. Then $\varphi\in L^1(0,\infty)$.
\end{theorem}

\begin{proof}
	
	By Lemma \ref{some characterizations of real Hardy spaces}, the following function
	$$f(x)=\frac{x}{(x^2+1)^2}, \quad x\in\R,$$
	is in $H^1(\R)$ since $f(x)\in L^1(\R)$ and $H(f)(x)= \frac{x^2-1}{2(x^2+1)^2}\in L^1(\R)$. Hence,
	$$\H_\varphi f(x)=\int_0^\infty \frac{\frac{x}{t}}{\left[(\frac{x}{t})^2+1\right]^2} \frac{\varphi(t)}{t} dt$$
	is in $H^1(\R)$ since $\H_\varphi$ is bounded on $H^1(\R)$. As a consequence,
	\begin{eqnarray*}
		\int_0^\infty \frac{y}{(y^2+1)^2}dy \int_0^\infty \varphi(t)dt &=& \int_0^\infty dx \int_0^\infty \frac{\frac{x}{t}}{\left[(\frac{x}{t})^2+1\right]^2} \frac{\varphi(t)}{t} dt\\
		&\leq& \|\H_\varphi f\|_{L^1(\R)}\lesssim \|\H_\varphi f\|_{H^1(\R)}<\infty,
	\end{eqnarray*} 
this implies that $\varphi\in L^1(0,\infty)$.
	
\end{proof}

{\bf Acknowledgements.} The paper was completed when the authors was visiting
to Vietnam Institute for Advanced Study in Mathematics (VIASM). We would like
to thank the VIASM for financial support and hospitality.

\end{document}